\documentclass[12pt]{amsart}
\usepackage{times,fullpage}
\usepackage{latexsym}
\usepackage{amssymb}
\usepackage{url}
\usepackage{pictexwd,dcpic}
\usepackage{graphicx}
\usepackage{pinlabel}

\newtheorem{thm}{Theorem}
\newtheorem{prop}{Proposition}
\newtheorem{cor}{Corollary}
\newtheorem{lem}{Lemma}

\theoremstyle{remark}
\newtheorem{rem}{Remark}

\newcommand{\CP}{\mathbb{CP}}
\newcommand{\Z}{\mathbb{Z}}

\DeclareMathOperator{\SO}{SO}
\DeclareMathOperator{\Or}{O}
\DeclareMathOperator{\SU}{SU}

\newcommand\rank{\operatorname{rank}}
\newcommand\Diff{\operatorname{Diff}}

\title[Branched coverings of simply connected manifolds]
      {Branched coverings of simply connected manifolds}
\author{Christoforos Neofytidis}
\address{Mathematisches Institut, {\smaller LMU} M\"unchen, Theresienstr.~39, 80333~M\"unchen, Germany}
\email{Christoforos.Neofytidis@mathematik.uni-muenchen.de}
\date{\today; \copyright{\ C.~Neofytidis 2012}}

\begin{document}

\begin{abstract}
We construct branched double coverings by certain direct products of manifolds for connected sums of copies of sphere bundles over the 2-sphere. As an application we answer a question of Kotschick and L\"oh up to dimension five. More precisely, we show that
\begin{itemize}
\item[(1)] every simply connected, closed four-manifold admits a branched double covering by a product of the circle with a connected sum of copies of
$S^2 \times S^1$, followed by a collapsing map;
\item[(2)] every simply connected, closed five-manifold admits a branched double covering by a product of the circle with a connected sum of copies of $S^3 \times S^1$, followed by a map whose degree is determined by the torsion of the second integral homology group of the target.
\end{itemize}
\end{abstract}
\maketitle


\section{Introduction}\label{s:intro}

The realization of manifolds as branched coverings is a classical long-standing problem in topology. A well-known theorem of Alexander~\cite{Al} states
that every oriented, closed, smooth $n$-dimensional manifold is a branched covering of $S^n$. Strong restrictions for the existence of branched coverings
were found by Berstein and Edmonds~\cite{BE}. Branched coverings have been investigated in many different contexts and they turned out 
to be a useful tool for the study of several problems in geometry, such as the minimal genus problem in four dimensions. 

Recall that a branched $d$-fold covering is a smooth proper map $f \colon X \longrightarrow Y$ with a codimension two subcomplex $B_f \subset Y$, called
the branch locus of $f$, such that $f \vert_{X \setminus f^{-1}(B_f)} \colon X \setminus f^{-1}(B_f) \longrightarrow Y \setminus B_f$ is a $d$-fold
covering in the usual sense and for each $x \in f^{-1}(B_f)$ the map $f$ is given by $(z,v) \mapsto (z^m,v)$, for some charts of $x$ and $f(x)$ and some
positive integer $m$. The point $x$ is called singular and its image $f(x)$ is called a branch point.

In dimensions two and three, Edmonds showed that a dominant map is quite often homotopic to a branched covering~\cite{Ed1,Ed2}. More precisely,
Edmonds proved that every non-zero degree map between two closed surfaces is homotopic to the composition of a pinch map followed by a branched
covering~\cite{Ed1}. A pinch map in dimension two is a map which collapses 2-handles, i.e. is a quotient map $\pi \colon \Sigma \longrightarrow
\Sigma/\Sigma'$, where a submanifold $\Sigma' \subset \Sigma$, with circle boundary in the interior of $\Sigma$, is identified to a point. In dimension
three, Edmonds result is that every $\pi_1$-surjective map of degree at least three is homotopic to a branched covering~\cite{Ed2}. The existence of branched coverings in low dimensions has been explored by several other people, including Fox, Hilden, Hirsch and Montesinos. 

Our aim in this paper is to show that every simply connected, closed four-manifold, resp. five-manifold, admits a branched double covering by a product
$S^1 \times N$, composed with a certain pinch map, resp. a map whose degree depends on the torsion of the integral homology of the target.

The main result is the following general statement:

\begin{thm}\label{t:main}
 For $n \geq 4$ and every $k$ there is a branched double covering
\begin{align*}
S^1 \times (\#_k S^{n-2} \times S^1) \longrightarrow \#_k (S^{n-2} \widetilde{\times} S^2),
\end{align*}
where $S^{n-2} \widetilde{\times} S^2$ denotes the total space of the non-trivial $S^{n-2}$-bundle over $S^2$ with structure group $\SO(n-1)$.
\end{thm}

We note that the statement of Theorem \ref{t:main} does not contain any reference on orientations, because the targets are simply connected and therefore
always orientable.

The existence of dominant maps, not necessarily branched coverings, where the domain is a non-trivial product has an independent interest and its study is
partially motivated by Gromov's work on functorial semi-norms on homology~\cite{gromovmetric}. Kotschick and L\"oh~\cite{KL} investigated such maps and
showed that many targets with suitably large fundamental groups are not dominated by products. On the other hand, the fundamental group conditions given
in~\cite{KL} are not always sufficient to deduce domination by products, cf.~\cite{KN,NeoThesis}.

However, as pointed out in~\cite{KL}, it is natural to ask whether every connected, oriented, closed manifold with finite fundamental group is dominated by a non-trivial product. In order to study this question, it obviously suffices to obtain an answer for simply connected targets. In dimensions two and three, the
answer is easy and affirmative, because $S^2$ and $S^3$ respectively represent the only simply connected manifolds in those two dimensions.

In this paper, we combine our constructions of branched coverings, with some classification results for simply connected manifolds in dimensions four and
five (cf.~\cite{Wall0,Fre} and~\cite{Sm,B} respectively) to obtain the following:

\begin{thm}\label{t:4and5-mfds}
 Every connected, oriented, closed manifold with finite fundamental group in dimensions four and five is dominated by a non-trivial product.
\end{thm}

In dimension four, Kotschick and L\"oh~\cite{KL} have previously obtained a non-constructive proof for the above statement based on a result of Duan and
Wang~\cite{DWa}. Our proof here is independent of those earlier works and it moreover gives an explicit construction of a degree two map from a product
of type $S^1 \times (\# S^2 \times S^1)$ to every simply connected four-manifold, obtained as the composition of a branched double covering with a
certain degree one map; cf. Theorem \ref{t:4-mfds}.

As we shall see in the course of the proof of Theorem \ref{t:4-mfds}, the (stable) diffeomorphism classification of Wall~\cite{Wall0} and Freedman~\cite{Fre} will give us the desired degree one map $\CP^2\#\overline{\CP^2} \longrightarrow M$. It is well-known that the homotopy classification of simply connected four-manifolds alone implies the existence of such a degree one map between the homotopy types of $\CP^2\#\overline{\CP^2}$  and $M$. However, we appeal to the (stable) diffeomorphism classification theorems in order to obtain a considerably stronger result where all our dominant maps are smooth (whenever $M$ itself is smooth). With those smooth maps in hand, we will be able to compare our construction with the aforementioned results of Edmonds in dimensions two and three; see the discussion following Remark \ref{r:Wall1} in Section \ref{s:4-mfds} and Section \ref{ss:final2}.

\subsection*{Outline} In Section \ref{s:construction}, we first prove Theorem \ref{t:main} and then we discuss briefly the
existence of branched double coverings for connected sums of direct products with a sphere factor. 
In Sections \ref{s:4-mfds} and \ref{s:5-mfds}, we apply the constructions of Section \ref{s:construction} to simply connected four- and five-manifolds respectively in order to
prove Theorem \ref{t:4and5-mfds}. In the final Section \ref{s:final}, we give another application in higher dimensions and we make further remarks.

\subsection*{Acknowledgments}
I am grateful to D. Kotschick for his support and inspiration and to P. Derbez, C. L\"oh, J. Souto and S. Wang for useful comments and discussions on an earlier draft of this paper. Also, I would like to thank the referee, in particular about the alternative argument in dimension five given in Remark \ref{r:5-mfdsandTan}. 

This work was funded by the {\em Deutscher Akademischer Austausch Dienst} (DAAD).

\section{Construction of branched double coverings}\label{s:construction}

We begin this section by proving Theorem \ref{t:main}. After that, we give a branched double covering for every connected sum of copies of a direct
product $M \times S^2$, where $M$ is any oriented, closed $n$-dimensional manifold (cf. Theorem \ref{t:general}).

\begin{figure}
\labellist
\pinlabel {\tiny $D^2=P(A)$} at 235 149
\pinlabel {\tiny $x_1$} at -5 91
\pinlabel {\tiny $x_2$} at 82 91
\pinlabel {\tiny $x_3$} at 82 -5
\pinlabel {\tiny $x_4$} at -5 -5
\pinlabel {\tiny $y_1$} at 355 137
\pinlabel {\tiny $y_2$} at 386 137
\pinlabel {\tiny $y_3$} at 391 23
\pinlabel {\tiny $y_4$} at 357 23
\pinlabel {$\stackrel{P}\longrightarrow$} at 235 87
\endlabellist
\centering
\includegraphics[width=8cm]{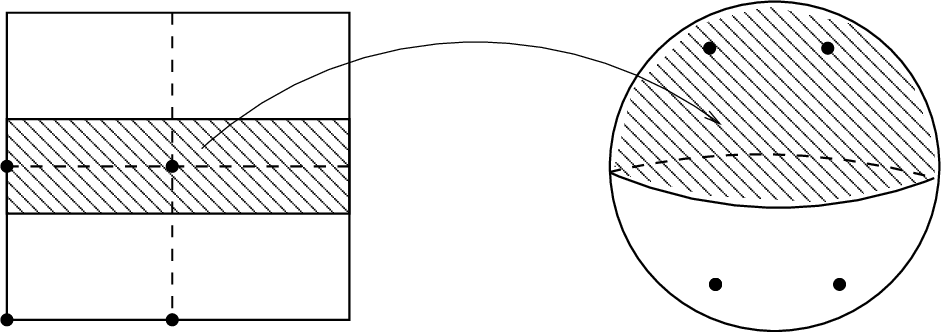}
\caption{\small The 2-torus is a branched double covering of the 2-sphere, $P \colon T^2 \longrightarrow S^2$, branched along four points $P(x_i) = y_i$, $i \in \{1,2,3,4\}$, and so that $P(A) = D^2$, where $A$ is an annulus in $T^2$ containing two singular points.}
\label{f:branchcovtor}
\end{figure}
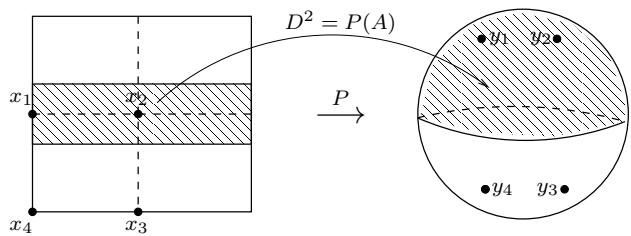

\begin{proof}[Proof of Theorem \ref{t:main}]
 By assumption, we consider oriented $S^{n-2}$-bundles over $S^2$ with structure group $\SO(n-1)$, where $n \geq 4$. These bundles are classified
by $\pi_1(\SO(n-1)) = \Z_2$ (cf. Steenrod~\cite{Steen}), therefore there exist only two; the product $S^{n-2} \times S^2$ and the twisted bundle $S^{n-2}
\widetilde{\times} S^2$. Let $\pi \colon S^{n-2} \widetilde{\times} S^2 \longrightarrow S^2$ denote the twisted bundle. There is a branched double covering
with four branch points, $P \colon T^2 \longrightarrow S^2$, given by the quotient for an involution on $T^2$; see Figure \ref{f:branchcovtor} (this map is 
known as ``pillowcase''). We pull back $\pi$ by $P$ to obtain a branched double covering $P^* \colon P^*(S^{n-2} \widetilde{\times} S^2) \longrightarrow S^{n-2} \widetilde{\times} S^2$. Now $P^*(S^{n-2} \widetilde{\times} S^2)$ is the total space of an oriented $S^{n-2}$-bundle over $T^2$ with structure group $\SO(n-1)$. Again, there exist
only two such bundles and since $P$ has even degree we deduce that $P^*(S^{n-2} \widetilde{\times} S^2)$ is the trivial bundle, i.e. the product $T^2
\times S^{n-2}$. Therefore, $T^2 \times S^{n-2}$ is a branched double covering of $S^2 \widetilde{\times} S^{n-2}$, with branch locus four copies of the
$S^{n-2}$-fiber of $S^2 \widetilde{\times} S^{n-2}$, given by the preimages under $\pi$ of the four branch points of $P$. This proves the statement of
Theorem \ref{t:main} for $k = 1$.

Next, we prove the claim for $k \geq 2$. Let the branched covering $P^* \colon T^2 \times S^{n-2} \longrightarrow S^2 \widetilde{\times}
S^{n-2}$ constructed above. We can think of $T^2 \times S^{n-2}$ as a trivial $S^1$-bundle over $S^1 \times S^{n-2}$. We thicken an $S^1$-fiber of this
bundle to an annulus $A$ in $T^2$ so that $P(A) = D^2$, as in Figure \ref{f:branchcovtor}. A fibered neighborhood of this $S^1$-fiber is a product $S^1
\times D^{n-1}$ in $T^2 \times S^{n-2}$ and $P^*(S^1 \times D^{n-1})$ is an $n$-disk $D^n$ in $S^2 \widetilde\times S^{n-2}$. We now remove the fibered
neighborhood $S^1 \times D^{n-1}$ from two copies of $T^2 \times S^{n-2}$ and perform a fiber sum by gluing the $S^1 \times S^{n-2}$-boundaries. Since $T^2
\times S^{n-2}$ is a trivial $S^1$-bundle this fiber sum will produce another trivial $S^1$-bundle, namely  a product $S^1 \times ((S^1 \times S^{n-2}) \#
(S^1 \times S^{n-2}))$. At the same time, we connected sum two copies of $S^2 \widetilde\times S^{n-2}$ along $D^n$, so that the branch loci fit together.
We have now obtained a branched double covering
\begin{align*}
 S^1 \times ((S^1 \times S^{n-2}) \# (S^1 \times S^{n-2})) \longrightarrow (S^2 \widetilde{\times} S^{n-2}) \# (S^2 \widetilde{\times} S^{n-2}),
\end{align*}
proving the claim for $k = 2$. For $k > 2$ we iterate the above construction.
\end{proof}

An immediate consequence of Theorem \ref{t:main} is the following:

\begin{cor}
For $n\geq 4$ and $k\geq 0$, every connected sum $\#_k(S^{2}\widetilde{\times}S^{n-2})$ is a quotient space of the direct product $S^1\times(\#_kS^{n-2}\times S^1)$.
\end{cor}

The statement of Theorem \ref{t:main} is the most general possible concerning oriented $S^{n-2}$-bundles over $S^2$, but it moreover presupposes the
structure group of these bundles to be linear. This assumption is not necessary in dimensions four and five (cf.~\cite{Smale} and~\cite{Ha} respectively)
and so we will be able to construct many branched coverings in those two dimensions; cf. Sections \ref{s:4-mfds} and \ref{s:5-mfds}.

The inspiration for the construction of the branched covering of Theorem \ref{t:main} stems from a previous joint work with Kotschick on domination for
three-manifolds by circle bundles~\cite{KN}. More precisely, we proved there that for every $k$ there is a $\pi_1$-surjective branched double covering $S^1
\times (\#_k S^1 \times S^1) \longrightarrow \#_k(S^1 \times S^2)$. A direct generalization of that construction is achieved if we replace $S^1$ in
$\#_k(S^1 \times S^2)$ by any oriented, closed $n$-dimensional manifold $M$. The steps of the proof follow those of~\cite{KN} and they are left to
the reader:

\begin{thm}\label{t:general}
 Let $M$ be a connected, oriented, closed $n$-dimensional manifold. For every $k$ there is a $\pi_1$-surjective branched double covering
$S^1 \times (\#_k M \times S^1) \longrightarrow \#_k (M \times S^2)$.
\end{thm}

\begin{rem} 
The constructions given in Theorems \ref{t:main} and \ref{t:general} are of different nature, because the connected
summands of the target in Theorem \ref{t:main} are twisted products, while in Theorem \ref{t:general} the summands are direct products. For this reason the
proof of Theorem \ref{t:main} cannot be deduced using Theorem \ref{t:general}. (Nevertheless, Theorem \ref{t:general} could be seen as
a special case of Theorem \ref{t:main} for $M = S^{n}$.) A further generalization of Theorem \ref{t:main} would be to consider targets that are connected
sums of fiber bundles over $S^2$, where the fiber is an arbitrary oriented, closed manifold $M$. However, the comprehension of arbitrary twisted products
$S^2 \widetilde{\times} M$ (and of connected sums built out of such summands) seems, in general, not sufficiently enough to produce further generalization
of Theorem \ref{t:main}.
\end{rem}

\begin{rem}\label{r:generalpillow}
The two-dimensional pillowcase $T^2=S^1\times S^1 \longrightarrow S^2$ can be generalized for any sphere $S^n, \ n\geq 2$. Namely, for every $n > k \geq 1$, there is a branched double covering $P \colon S^k \times S^{n - k} \longrightarrow S^n$. The branched locus of $P$ is $B_P = S^{k-1} \times S^{n-k-1}$ and there is an $n$-ball $D^n \subset S^n$ such that $P^{-1}(D^n) = S^k \times D^{n-k}$. Therefore, in place of $S^2$ in Theorem \ref{t:general} we may take any sphere of dimension at least two. We refer to Chapter 3 of~\cite{NeoThesis} for further details.
\end{rem}



\section{Simply connected four-manifolds}\label{s:4-mfds}

We now apply Theorem \ref{t:main} to simply connected four-manifolds. First, we show that every four-manifold diffeomorphic to a connected sum $\#_k \CP^2
\#_k \overline{\CP^2}$ admits a branched double covering by a product $S^1 \times (\#_k S^1 \times S^2)$. We then deduce by the classification results of
Wall~\cite{Wall0} and Freedman~\cite{Fre}, that every simply connected, closed four-manifold admits a branched double covering by a
product of the circle with a connected sum of copies of $S^2 \times S^1$, followed by a collapsing map.

By a result of Smale~\cite{Smale}, the inclusion $\SO(3) \hookrightarrow \Diff^+(S^2)$ is a homotopy equivalence and so there exist only two oriented
$S^2$-bundles over $S^2$. Moreover, $\CP^2 \#\overline{\CP^2}$ is diffeomorphic to $S^2 \widetilde{\times} S^2$ (cf. Wall~\cite{Wall}), therefore Theorem
\ref{t:main} implies the following for $n = 4$:

\begin{cor}\label{c:4-mfds}
 For every $k$ there is a branched double covering $S^1 \times (\#_k S^1 \times S^2) \longrightarrow \#_k \CP^2 \#_k \overline{\CP^2}$.
\end{cor}

Given a connected sum $P \# Q$, we define a pinch map $P \# Q \longrightarrow P$, by collapsing the gluing sphere and then mapping $Q$ to a point. This
degree one map is called collapsing map. For a connected sum $P \# \overline{\CP^2}$, instead of the whole summand $\overline{\CP^2}$, we may collapse the
exceptional embedded sphere $\overline{\CP^1} \subset \overline{\CP^2}$ to obtain again a degree one map $P \# \overline{\CP^2} \longrightarrow P$. The
collapsed $\overline{\CP^1}$ has self-intersection number $-1$ and is called $-1$-sphere. If $P$ is smooth, then this is the usual blow-down operation.
Similarly, for a connected sum $P \# \CP^2$ we obtain a degree one map $P \# \CP^2 \longrightarrow P$ by collapsing the embedded $+1$-sphere $\CP^1 \subset
\CP^2$. In the smooth category, this operation is known as the antiblow-down of $P$.

In the light of Corollary \ref{c:4-mfds}, we now rely on results of Wall~\cite{Wall0} and Freedman~\cite{Fre} on the classification of simply connected,
closed four-manifolds to obtain the following statement which moreover proves Theorem \ref{t:4and5-mfds} in dimension four:

\begin{thm}\label{t:4-mfds}
Every simply connected, closed four-manifold $M$ admits a degree two map by a product $S^1 \times (\#_k S^2 \times S^1)$, which
is given by the composition of a branched double covering $S^1 \times (\#_k S^2 \times S^1) \longrightarrow \#_k \CP^2 \#_k \overline{\CP^2}$ with a
collapsing map $\#_k \CP^2 \#_k \overline{\CP^2} \longrightarrow M$.
\end{thm}
\begin{proof}
 By Corollary \ref{c:4-mfds}, it suffices to show that for every simply connected, closed four-manifold $M$ there exists a $k$ and a collapsing map
$\#_k \CP^2 \#_k \overline{\CP^2} \longrightarrow M$.

First, if necessary, we perform connected sums of $M$ with copies of $\CP^2$ or $\overline{\CP^2}$ (or both) to obtain a manifold $M \#_p \CP^2 \#_q
\overline{\CP^2}$, whose intersection form is odd and indefinite (and therefore diagonal; cf.~\cite{MH}).

If $M$ is smooth, then Wall's stable diffeomorphism classification in dimension four~\cite{Wall0} implies that $M\#_p \CP^2 \#_q \overline{\CP^2}$ is
stably diffeomorphic to a connected sum $\#_l \CP^2 \#_m \overline{\CP^2}$. The connected summing with $\overline{\CP^2}$, resp. $\CP^2$, is the blow-up,
resp. antiblow-up, operation.

For $M$ non-smooth, we may first assume that the homotopy type of $M$ has trivial Kirby-Siebenmann invariant 
and then deduce that $M\#_p \CP^2 \#_q \overline{\CP^2}$ is homeomorphic to a connected sum $\#_l \CP^2 \#_m
\overline{\CP^2}$, by Freedman's topological classification of simply connected four-manifolds~\cite{Fre}. In particular, $M\#_p \CP^2 \#_q
\overline{\CP^2}$ inherits a smooth structure.

We can now assume that, in both cases, $M\#_p \CP^2 \#_q \overline{\CP^2}$ is (stably) diffeomorphic to a connected sum $\#_k \CP^2 \#_k \overline{\CP^2}$,
after connected summing with more copies of $\CP^2$ or $\overline{\CP^2}$, if necessary.

Finally, a degree one collapsing map
\begin{align*}
\#_k \CP^2 \#_k \overline{\CP^2} \cong M \#_p \CP^2 \#_q \overline{\CP^2} \longrightarrow M
\end{align*}
is obtained by collapsing the $q$ embedded exceptional spheres $\overline{\CP^1} \subset \overline{\CP^2}$ and the $p$ embedded spheres $\CP^1
\subset \CP^2$. If $M$ is smooth, then the collapsing map is also smooth. 

We have now completed the proof of Theorem \ref{t:4-mfds} and therefore of Theorem \ref{t:4and5-mfds} in dimension four.
\end{proof}

\begin{rem}\label{r:Wall1}
 Wall's stable diffeomorphism classification~\cite{Wall0} is obtained after adding summands $S^2 \times S^2$. However, for simply connected, closed, smooth
four-manifolds with odd intersection form (as in the proof of Theorem \ref{t:4-mfds}), this is equivalent to adding summands $\CP^2 \# \overline{\CP^2}$,
again by a result of Wall~\cite{Wall}; see also Remark \ref{r:Wall2}.
\end{rem}

In dimension two, Edmonds~\cite{Ed1} proved that every non-zero degree map between two closed surfaces is homotopic to the composition of a pinch map
followed by a branched covering. We observe that the maps constructed in Theorem \ref{t:4-mfds} for every simply connected, closed four-manifold have an
analogy to Edmonds result, where however the order between the pinch map and the branched covering is reversed.

Moreover, Edmonds~\cite{Ed1} showed that a non-zero degree map between two closed surfaces, $f \colon \Sigma \longrightarrow F$, is homotopic to a branched
covering if and only if either $f$ is $\pi_1$-injective or $|\deg(f)| > [\pi_1(F) \colon \pi_1(f)(\pi_1(\Sigma))]$. This implies that every non-zero degree
map between two closed surfaces can be lifted to a ($\pi_1$-surjective) map which is homotopic either to a pinch map (absolute degree one) or to a branched
covering.

In dimension three, every $\pi_1$-surjective map of degree greater than two between two closed three-manifolds is homotopic to a branched covering, again
by a result of Edmonds~\cite{Ed2}. This result fails in dimension four, by a recent work of Pankka and Souto~\cite{PS}, where it is shown that $T^4$ is not
a branched covering of $\#_3 (S^2 \times S^2)$, while every integer can be realized as the degree for a map from $T^4$ to $\#_3 (S^2 \times S^2)$ (by a criterion 
of Duan and Wang~\cite{DWa}; see Section \ref{ss:final2}).

\section{Simply connected five-manifolds}\label{s:5-mfds}

In this section we show that every closed five-manifold with finite fundamental group is dominated by products. We first recall the classification of
simply connected, closed five-manifolds by Smale~\cite{Sm} and Barden~\cite{B} and show that every five-manifold diffeomorphic to a connected sum of copies
of the two $S^3$-bundles over $S^2$ admits a branched double covering by a product $S^1 \times (\#_k S^1 \times S^3)$. We then give some existence results
of dominant maps between simply connected five-manifolds and using these results we prove Theorem \ref{t:4and5-mfds} for five-manifolds.

Given two $n$-dimensional manifolds $M$ and $N$ with boundaries $\partial M$ and $\partial N$ respectively, we form a new manifold
$M \cup_fN$, where $f$ is an orientation reversing diffeomorphism of any $(n-1)$-dimensional submanifold of $\partial N$ with one of $\partial M$.

Smale~\cite{Sm} classified simply connected, spin five-manifolds and a few years later Barden~\cite{B} completed that classification by including
non-spin manifolds as well.
The following constructions are given in~\cite{B}: Let $S^3 \times S^2$, $S^3 \widetilde{\times} S^2$ be the two $S^3$-bundles over $S^2$ and
$A = S^2 \times D^3$, $B = S^2 \widetilde{\times} D^3$ be the two $D^3$-bundles over $S^2$ with boundaries $\partial A = S^2 \times S^2$ and $\partial B = S^2\widetilde{\times} S^2 \cong \CP^2 \# \overline{\CP^2}$ respectively. As in dimension four, we don't need to assume that the structure group of oriented $S^3$-bundles over $S^2$ is linear, because the inclusion $\SO(4) \hookrightarrow \Diff^+(S^3)$ is a homotopy equivalence, by the proof of the Smale conjecture (cf. Hatcher~\cite{Ha}).

A prime, simply connected, closed, spin five-manifold is either $M_1 := S^5$ or $M_\infty := S^2 \times S^3$, if its integral homology groups have no
torsion. If its second homology group is torsion, then is
\begin{align*}
M_k := (A \#_{\partial} A) \cup_{\overline{f_k}} (\overline{A \#_{\partial} A}), 1 < k < \infty,
\end{align*}
where $A \#_{\partial} A$ denotes the boundary connected sum of two copies of $A$ and $f_k$ is an orientation preserving diffeomorphism on
$\partial (A \#_{\partial} A) = (S^2 \times S^2) \# (S^2 \times S^2)$, realizing a certain isomorphism
\begin{align*}
 H_2(f_k;\Z) \colon H_2 (\partial (A \#_{\partial} A); \Z) \stackrel{\cong} \longrightarrow H_2 (\partial (A \#_{\partial}A);\Z).
\end{align*}
The second integral homology of $M_k$ is given by $H_2 (M_k; \Z) = \Z_{k} \oplus \Z_{k}$, $1 < k < \infty$; cf.~\cite{B}. For details on the
$f_k$'s construction see~\cite{Wall}.

A prime, simply connected, closed, non-spin five-manifold with torsion-free integral homology is the non-trivial $S^3$-bundle over $S^2$, denoted by
$X_\infty$. A simply connected, closed, non-spin five-manifold with torsion second integral homology is either
\begin{itemize}
 \item $X_{-1} := B \cup_{\overline{g_{-1}}} \overline{B}$,
where $g_{-1}$ is an orientation preserving diffeomorphism on $\partial B$, realizing a certain isomorphism
\begin{align*}
 H_2 (g_{-1};\Z) \colon H_2 (\partial B; \Z) \stackrel{\cong} \longrightarrow H_2 (\partial B; \Z),
\end{align*}
or
\item $X_m := (B \#_{\partial} B) \cup_{\overline{g_m}} (\overline{B \#_{\partial} B})$, $1 \leq m < \infty$, where $B \#_{\partial} B$ denotes the
boundary connected sum of two copies of $B$ and $g_m$ is an orientation preserving diffeomorphism on
$\partial (B \#_{\partial} B) = (S^2 \widetilde{\times} S^2) \# (S^2 \widetilde{\times} S^2)$, realizing a certain isomorphism
\begin{align*}
 H_2(g_m;\Z) \colon H_2 (\partial (B \#_{\partial} B); \Z) \stackrel{\cong} \longrightarrow H_2 (\partial(B \#_{\partial} B);\Z).
\end{align*}
\end{itemize}
Except $X_1 \cong X_{-1} \# X_{-1}$, each $X_m$ is prime. Their second integral homology group is given by $H_2 (X_m; \Z) = \Z_{2^m} \oplus
\Z_{2^m}$, $0 < m < \infty$ and $H_2 (X_{-1}; \Z) = \Z_2$; cf.~\cite{B}. Also, $X_{-1}$ is the Wu manifold $\SU (3) / \SO(3)$; cf.~\cite{Do,Wu}. Finally,
we set $X_0 := S^5$ which is spin and has torsion-free homology.

According to the above data, we have the following classification theorem for simply connected five-manifolds.

\begin{thm}[Barden~\cite{B}]\label{t:Barden}
Every simply connected, closed five-manifold $M$ is diffeomorphic to a connected sum $M_{k_1} \# ... \# M_{k_l}
\# X_m$, where $-1 \leq m \leq \infty$, $l \geq 0$, $k_1 > 1$ and $k_i$ divides $k_{i+1}$ or $k_{i+1} = \infty$.
\end{thm}

Since every summand, except $X_1 \cong X_{-1} \# X_{-1} $, is prime, we may refer to the above decomposition of $M$ as prime decomposition. Moreover, a
summand $X_{m \neq 0}$ exists if and only if $M$ is not spin.

For a simply connected, closed five-manifold with torsion-free homology the following particular classification result holds:

\begin{thm}[Smale~\cite{Sm}, Barden~\cite{B}]\label{t:BaSm}
Every simply connected, closed five-manifold $M$ with torsion-free second homology group $H_2(M;\Z)=\Z^k$ is (up to diffeomorphism)
\begin{enumerate}
\item either a connected sum $\#_k (S^3 \times S^2)$, for $M$ spin, or
\item a connected sum $\#_{k-1} (S^3 \times S^2) \# (S^3 \widetilde{\times} S^2)$, for $M$ non-spin.
\end{enumerate}
\end{thm}

\begin{rem}\label{r:Wall2}
 By a theorem of Wall~\cite{Wall}, a connected sum $\#_{k-1} (S^3 \times S^2) \# (S^3 \widetilde{\times} S^2)$ is diffeomorphic to $\#_k
(S^3 \widetilde{\times} S^2)$. In Remark \ref{r:Wall1}, we refer to the corresponding statement in dimension four. This is also true for all oriented
$S^{n-2}$-bundles over $S^2$, $n \geq 6$, with structure group $\SO(n-1)$.
\end{rem}

The above classification results will be our guide to prove Theorem \ref{t:4and5-mfds} for five-manifolds. First, we can prove this theorem for every
simply connected five-manifold with torsion-free second homology group:

\begin{prop}\label{p:torsionfree}
  Let $M$ be a simply connected, closed five-manifold with torsion-free second homology group $H_2(M;\Z)=\Z^k$. Then $M$ admits a branched double covering
by the product $S^1 \times (\#_k S^1 \times S^3)$.
\end{prop}
\begin{proof}
 We know by Theorem \ref{t:BaSm} and by Remark \ref{r:Wall2} that such $M$ is diffeomorphic to a connected sum of copies of the twisted product $S^2
\widetilde{\times} S^3$ or of the trivial bundle $S^3 \times S^2$. Now Theorem \ref{t:main} and Theorem \ref{t:general} imply the proof, for $\#_k (S^2
\widetilde{\times} S^3)$ and $\#_k (S^3 \times S^2)$ respectively.
\end{proof}

In the remainder of this section we complete the proof of Theorem \ref{t:4and5-mfds} for five-manifolds. We observe that $S^2 \times S^3$ is a branched
double covering of the other two prime, simply connected, closed five-manifolds with torsion-free homology, namely of $S^2 \widetilde{\times} S^3$ and
$S^5$. Indeed:
\begin{itemize}
\item A branched double covering $S^2 \times S^3 \longrightarrow S^2 \widetilde{\times} S^3$ is obtained by pulling back the $S^3$-bundle map $S^2
\widetilde{\times} S^3 \longrightarrow S^2$ by a branched double covering $S^2 \longrightarrow S^2$ (recall that $S^3$-bundles over $S^2$ are classified by
$\pi_1(\SO(4)) = \Z_2$).
\item $S^2 \times S^3$ is a branched double covering of $S^5$, by pulling back the $S^1$-bundle map $S^5 \longrightarrow
\CP^2$ by a branched double covering $S^2 \times S^2 \longrightarrow \CP^2$; cf.~\cite{Gibl,DuL}. (A branched double covering $S^2 \times S^2 \longrightarrow \CP^2$ is obtained as the quotient map for the involution $(x,y) \mapsto (y,x)$ of $S^2\times S^2$.)
\end{itemize}
As we shall see below, every simply connected, closed
five-manifold $M$ with torsion second integral homology group admits a non-zero degree map by $S^2 \times S^3$, given by the composition of a branched
double covering $S^2 \times S^3 \longrightarrow S^5$ with a dominant map $S^5 \longrightarrow M$ whose degree depends on $H_2(M;\Z)$. The latter map will
be obtained by applying the Hurewicz theorem modulo a Serre class of groups.

A Serre class of abelian groups is a non-empty class $\mathcal{C}$ of abelian groups such that for every exact sequence
$A \longrightarrow B \longrightarrow C$, where $A, C \in \mathcal{C}$, then $B\in\mathcal{C}$. A Serre class $\mathcal{C}$ is called a ring of
abelian groups if it is closed under the tensor and the torsion product operations. Moreover, $\mathcal{C}$ is said to be acyclic if for every aspherical
space $X$ with $\pi_1 (X) \in \mathcal{C}$, then the homology groups $H_i (X;\Z) \in \mathcal{C}$, for all $i > 0$. We say that two abelian groups $A, B$ are
isomorphic modulo $\mathcal{C}$ if there is a homomorphism between $A$ and $B$ whose kernel and cokernel belong to $\mathcal{C}$.

The Hurewicz theorem modulo a Serre class states the following:

\begin{thm}[Serre~\cite{Se}]\label{t:Serre}
 Let $X$ be a simply connected space and $\mathcal{C}$ be an acyclic ring of abelian groups. Then the following are equivalent:
 \begin{itemize}
  \item $\pi_i (X) \in \mathcal{C}$, for all $1 < i < n$,
  \item $H_i (X) \in \mathcal{C}$, for all $1 < i < n$.
 \end{itemize}
 Moreover, each of the above statements implies that the Hurewicz homomorphism $h \colon \pi_i(X) \longrightarrow H_i (X)$ is an isomorphism modulo
$\mathcal{C}$ for all $i \leq n$.
\end{thm}

As a consequence of this version of the Hurewicz theorem, every simply connected, closed $n$-dimensional manifold $M$ whose homology groups $H_i (M;\Z)$
are all $k$-torsion, for $0 < i < n$, is minimal with respect to the domination relation (i.e. it is dominated by every other manifold).

\begin{cor}[Ruberman~\cite{Ru}]\label{c:Ru}
 Let $M$ be a simply connected, closed $n$-dimensional manifold with $k$-torsion homology groups $H_i(M;\Z)$, for $0 < i < n$ and some integer $k$. Then
the image of the Hurewicz homomorphism $\pi_n (M) \longrightarrow H_n (M)$ is given by $k^r \Z$ for some $r$. In particular, there is a map $S^n
\longrightarrow M$ of degree $k^r$.
\end{cor}

We have now shown that every five-manifold $M$ which is a connected sum of copies of $M_k$ and $X_m$, where $1 \leq k < \infty$ and $-1 \leq m < \infty$,
admits a branched double covering by the product $S^2 \times S^3$, followed by a map $S^5 \longrightarrow M$ whose degree is determined by $H_2(M;\Z)$,
being a power of the least common multiple of the torsion second integral homology groups $H_2(M_k;\Z)$ and $H_2(X_m;\Z)$. As we have seen
above, the non-trivial $S^3$-bundle over $S^2$
admits a branched double covering by the product $S^2 \times S^3$. We now want to combine these maps, together with our constructions from Section
\ref{s:construction}, to obtain domination by products for every simply connected, closed five-manifold.

We remark that it is not generally possible to obtain a non-zero degree map $M_1 \# M_2 \longrightarrow N_1 \# N_2$ by connected summing any two non-zero
degree maps between closed $n$-dimensional manifolds, $f_i \colon M_i \longrightarrow N_i$, $i = 1,2$.  A first obstruction is that the preimage
$f_i^{-1}(D_i^n)$ of a removed $n$-disk $D_i^n$ from $N_i$, is not generally an $n$-disk in $M_i$. Moreover, even if we can find such a disk, we need the
degrees of $f_1$ and $f_2$ to be equal in order to connected sum the domains and the targets, otherwise we cannot paste those maps along the gluing sphere
(recall that maps between spheres are classified by their degrees).

If the above two constraints are satisfied, then we can paste those $f_i$ together to obtain a new map $M_1 \# M_2 \longrightarrow N_1 \# N_2$ of the same
non-zero degree. The $\pi_1$-surjectivity of the $f_i$ is a sufficient condition to overcome the first obstacle (cf.~\cite{St}).

\begin{lem}[Derbez-Sun-Wang~\cite{DSW}]\label{l:consum}
 Let $M_i$, $N_i$ be connected, oriented, closed $n$-dimensional manifolds, $n \geq 3$, and assume that for $i = 1,...,k$ there exist $\pi_1$-surjective
maps $M_i \longrightarrow N_i$ of non-zero degree $d$. Then there is a $\pi_1$-surjective map $\#_{i=1}^{k} M_i \longrightarrow \#_{i=1}^{k} N_i$ of degree
$d$.
\end{lem}

For simply connected targets the $\pi_1$-surjectivity condition is automatically satisfied. Moreover, $S^2 \times S^3$ and $S^5$ admit self-maps of any
degree, which implies that the domination for every minimal summand by $S^5$ and for every $S^3$-bundle over $S^2$ by $S^2 \times S^3$ can be done by maps
of the same degree. We therefore obtain the following statement which also completes the proof of Theorem \ref{t:4and5-mfds} for five-manifolds:

\begin{thm}\label{t:5-mfds}
 Let $M$ be a simply connected, closed five-manifold so that $H_2(M;\Z)$ has rank $k$. Then $M$ admits a non-zero degree map by the product $S^1 \times
(\#_k S^1
\times S^3)$, which is given by the composition of a branched double covering $S^1 \times (\#_k S^1 \times S^3) \longrightarrow \#_k(S^2 \times S^3)$ with
a map $\#_k(S^2 \times S^3) \longrightarrow M$ whose degree is determined by the torsion of $H_2(M;\Z)$.
\end{thm}
\begin{proof}
 By Theorem \ref{t:Barden}, a simply connected, closed five-manifold $M$ is diffeomorphic to a connected sum $M_{k_1} \# ... \# M_{k_l} \# X_m$, where the
summands $M_{k_i}$, $X_m$ are described at the beginning of this section. Clearly, the rank $k$ of $H_2(M;\Z)$ is equal to the number of $M_\infty$ and
$X_\infty$, i.e. the number of $S^3$-bundles over $S^2$. Furthermore, we may assume that the torsion of $H_2(M;\Z)$
is not trivial, otherwise we appeal to Proposition \ref{p:torsionfree} to deduce that $M$ admits a branched double covering by a product $S^1 \times (\#_k
S^1 \times S^3)$.

By Corollary \ref{c:Ru} (and the comments following that) we deduce that the part of $M_{k_1} \# ... \# M_{k_l} \# X_m$ not containing $M_{\infty}$ or
$X_\infty$ admits a dominant map by $S^5$ whose degree is a power of the least common multiple of the torsion second integral homology groups
$H_2(M_{k_i};\Z)$ and $H_2(X_m;\Z)$, $k_i, m \neq \infty$. Now Lemma
\ref{l:consum} implies that $M$ is dominated by a manifold diffeomorphic to $\#_k (S^2 \times S^3)$ (if $m = \infty$, we additionally use the fact that
$M_\infty$ is a branched double covering of $X_\infty$, as we have seen above in this section). The degree of $\#_k(S^2 \times S^3) \longrightarrow M$ is
clearly determined (up to multiplication by two) by the torsion of $H_2(M;\Z)$. Finally, Theorem \ref{t:general} implies that $\#_k (S^2 \times S^3)$
admits a branched double covering by $S^1 \times (\#_k S^1 \times S^3)$, completing the proof of Theorem \ref{t:5-mfds}.
\end{proof}

\begin{rem}
By Remark \ref{r:generalpillow}, we also have that $\#_k (S^2 \times S^3)$ admits branched double coverings by $S^1 \times (\#_k S^2 \times S^2)$ and $S^2 \times (\#_k S^2 \times S^1)$, and branched four-fold coverings by $T^2 \times (\#_k S^1 \times S^2)$ and $S^1 \times S^2 \times \Sigma_k$, where $\Sigma_k$ is an oriented, closed surface of genus $k$.
\end{rem}

\begin{rem}\label{r:mappingdegrees}
In contrast to Theorem \ref{t:4-mfds}, the statement of Theorem \ref{t:5-mfds} does not provide absolute control of the degree of the map $\#_k(S^2 \times S^3)
\longrightarrow M$ following the branched double covering $S^1 \times (\#_k S^1 \times S^3) \longrightarrow \#_k(S^2 \times S^3)$. This is related to the
problem of determining the sets of (self-)mapping degrees of simply connected five-manifolds, which is essentially still unsolved; cf.~\cite{Tan}. However, an aside application of Theorem \ref{t:Serre} is that the sets of degrees of maps between the summands $M_k$ and $X_m$ (for $k,m \neq \infty$) are infinite, because every multiple
of a power of the torsion second homology of $X_m$ (resp. $M_k$) can be realized as a degree for a map $M_k \longrightarrow X_m$ (resp. $X_m \longrightarrow M_k$); recall that $S^5$ admits self-maps of any degree and it is minimal for the domination relation. In particular, we deduce, without using formality (see~\cite{CL}), that the set
of self-mapping degrees for every simply connected five-manifold is infinite.
\end{rem}

\begin{rem}\label{r:5-mfdsandTan}
Following the ideas of~\cite{Tan}, one can show the existence of a non-zero degree map $\#_k(S^2\times S^3) \longrightarrow M$ using the cell decompositions of $M$ and $\#_k(S^2\times S^3)$. Namely, if $M$ is a simply connected five-manifold with $\rank H_2(M;\Z)=k$, then by~\cite{B} there is a simply connected three-dimensional cell complex $X^3$ with torsion second integral homology group such that 
\[
M=(S^2_1\vee S^3_1\vee\cdots\vee S^2_k\vee S^3_k\vee X^3)\cup_{\varphi}D^5.
\] 
The attaching map $\varphi$ is given as the sum of the Whitehead products $[\iota_i,\eta_i]$ (where $\iota_i$ and $\eta_i$ denote inclusions from $S^2$ and $S^3$ respectively) with a summand of finite order $d$ which depends on the torsion of $H_2(M;\Z)$. Moreover, the connected sum $\#_k(S^2\times S^3)$ has a cell decomposition
\[
\#_k(S^2\times S^3)=(S^2_1\vee S^3_1\vee\cdots\vee S^2_k\vee S^3_k)\cup_{\psi}D^5,
\]
where $\psi$ is the sum of the Whitehead products $[\iota_i,\eta_i]$. Let now the map
\[
f\colon S^2_1\vee S^3_1\vee\cdots\vee S^2_k\vee S^3_k \longrightarrow S^2_1\vee S^3_1\vee\cdots\vee S^2_k\vee S^3_k\vee X^3
\]
given by $f\vert_{S^2_i}=d\cdot\iota_i$ and $f\vert_{S^3_i}=\eta_i$. Then for all $i=1,...,k$ we obtain
\[
f_*[\iota_i,\eta_i]=[f_*\iota_i,f_*\eta_i]=[d\cdot\iota_i,\eta_i]=d\cdot[\iota_i,\eta_i]
\]
and so $f_*(\psi)=d\cdot\varphi$, implying that $f$ extends to a map $\#_k(S^2\times S^3) \longrightarrow M$ of degree $d$; see~\cite{Tan} for further details.
\end{rem}

\section{A further application and final remarks}\label{s:final}

\subsection{Six-manifolds}\label{ss:final1}

In the light of Theorems \ref{t:main} and \ref{t:general} we can further verify that the fundamental classes of certain simply connected manifolds in
dimensions higher than five are representable by products.

As an example we deal with 2-connected, closed six-manifolds. Wall~\cite{Wall1} classified simply connected, closed, smooth six-manifolds (see
also~\cite{Zhubr}). As usual, the empty connected sum is $S^6$.

\begin{thm}[Wall~\cite{Wall1}]
 Let $M$ be a simply connected, closed, smooth six-manifold. Then $M$ is diffeomorphic to $N \# (S^3 \times S^3) \# \cdots \# (S^3 \times S^3)$, where
$H_3(N)$ is finite.
\end{thm}

If now the target is 2-connected then $N = S^6$; cf. Smale~\cite{Sm}. In that case the topological and the diffeomorphism
classification coincide (cf. Wall~\cite{Wall1}) and so we obtain:

\begin{cor}\label{c:6-manifolds}
 Every 2-connected, closed six-manifold $M$ admits a branched double covering by $S^1 \times (\#_k S^2 \times S^3)$, where  $k=\frac{1}{2}\rank H_3(M;\Z)$.
\end{cor}
\begin{proof}
 By the above results of Wall~\cite{Wall1} and Smale~\cite{Sm}, every 2-connected, closed six-manifold $M$ is diffeomorphic to $\#_k (S^3 \times S^3)$, 
for some non-negative integer $k = \frac{1}{2} \rank H_3 (M; \Z)$. The proof now follows by the generalized pillowcase map in Remark \ref{r:generalpillow}.
\end{proof}

\begin{rem}
By Remark \ref{r:generalpillow}, $M$ admits also a branched double covering by $S^2 \times (\#_k S^1 \times S^3)$, and branched four-fold coverings by $T^2 \times (\#_k S^2 \times S^2)$, $S^1 \times S^2 \times (\#_k S^2 \times S^1)$ and $S^2 \times S^2 \times \Sigma_k$.

We note that $M$ admits a degree two map by $S^1\times(\#_k T^5)$ as well. For take the branched double covering $S^1 \times (\#_k S^2 \times S^3) \longrightarrow M$ of Corollary \ref{c:6-manifolds} and precompose it with a degree one map $S^1 \times (\#_k T^2 \times T^3) \longrightarrow S^1 \times (\#_k S^2 \times S^3)$, using Lemma \ref{l:consum} and the fact that every sphere is minimal admitting maps of any degree.
\end{rem}

\subsection{Branched coverings in dimension four}\label{ss:final2}

Every dominant map onto a simply connected four-manifold
\begin{align}\label{eq.1}
 S^1 \times (\#_k S^2 \times S^1) \longrightarrow N,
\end{align}
is a priori $\pi_1$-surjective because the target has trivial fundamental group. However, the branched double covering $S^1 \times (\#_k S^2
\times S^1) \longrightarrow \#_k (S^2 \times S^2)$ obtained in Theorem \ref{t:general} (for $M = S^2$) is $\pi_1$-surjective also by construction. In a
previous joint work with Kotschick~\cite{KN} we have shown that there is a branched double covering
\begin{align}\label{eq.2}
 S^1 \times \Sigma_k \longrightarrow \#_k (S^2 \times S^1)
\end{align}
which is again $\pi_1$-surjective (now the target $\#_k (S^2 \times S^1)$ is not simply connected). We multiply (\ref{eq.2}) with the identity map on
$S^1$ to obtain a $\pi_1$-surjective branched double covering
\begin{align}\label{eq.3}
 T^2 \times \Sigma_k \longrightarrow S^1 \times (\#_k S^2 \times S^1).
\end{align}

Now we compose (\ref{eq.3}) with (\ref{eq.1}) to obtain a dominant map $T^2 \times \Sigma_k \longrightarrow N$. This map is not a branched covering,
not even in the case where $N$ is diffeomorphic to $\#_k \CP^2 \#_k \overline{\CP^2}$, because the branch locus
of (\ref{eq.3}) intersects the preimage of the branch locus of (\ref{eq.1}) in a codimension four subset of $S^1 \times (\#_k S^2 \times S^1)$. 

Nevertheless, we note that $T^2 \times \Sigma_k$ is a branched four-fold cover of $\#_k(S^2 \times S^2)$, by Remark \ref{r:generalpillow}. The existence of dominant maps from products $T^2 \times \Sigma_k$ to every simply connected four-manifold has also been shown by Kotschick and L\"oh~\cite{KL}, using a result of Duan and Wang~\cite{DWa}. That result states that if $X$ and $Y$ are oriented, closed four-manifolds and $Y$ is simply connected, then a degree $d \neq 0$ map $f \colon X \longrightarrow Y$ exists if and only if the intersection form of $Y$, multiplied by $d$, is embedded into the intersection form of $X$, where the embedding is given by $H^*(f)$ (the ``only if'' part is obvious).

As we mention in Section \ref{s:4-mfds}, the above existence criterion by Duan and Wang implies that every integer can be realized as the degree of a map from $T^4$ to $\#_3 (S^2 \times S^2)$, although no such map can be deformed to a branched covering by a recent result of Pankka and Souto~\cite{PS}. However, according to Theorem \ref{t:4-mfds}, it is natural to ask when a $\pi_1$-surjective non-zero degree map between two connected, oriented, closed four-manifolds is homotopic to the composition of a branched covering with a pinch map.  
Another natural question in this context is the following (suggested to me by J. Souto): Let $N$ be a connected, oriented, closed four-manifold and $f \colon M \longrightarrow N$ a $\pi_1$-surjective map of non-zero degree. Is there a $k \geq 0$ so that the composition of the pinch map $\#_k(S^2 \times S^2) \# M \longrightarrow M$ with $f$ is homotopic to a branched covering?

\subsection{Domination by products and sets of self-mapping degrees}\label{ss:final3}

Kotschick and L\"oh~\cite{KL} asked whether every closed manifold with finite
fundamental group is dominated by a product. In this paper, we have answered in positive that question in dimensions four and
five and in certain higher dimensions. For all our dominant maps onto a simply connected $n$-dimensional
manifold, the domain can be taken to be a product of the circle with a connected sum of copies of $T^{n-1}$. According to this, one could ask
whether every closed manifold with finite fundamental group is dominated by a product of type $S^1 \times N$, where $N$ is an $(n-1)$-dimensional manifold
(e.g. $N = \#_k T^{n-1}$ for some $k \geq 0$), or more generally, by a product of type $S^m \times N$. This is a considerably stronger question than that of~\cite{KL}, however, 
it seems less likely to be true. For instance, simply connected manifolds which admit self-maps of absolute degree at
most one might not be dominated by products (at least) of type $S^m \times N$. Such manifolds exist and are called inflexible; see for example~\cite{CL}. In this direction, and since the set of self-mapping degrees of $S^m \times N$ is infinite, we ask the following (at least in the simply connected case): Are there examples of inflexible manifolds that are dominated by flexible ones (i.e. by manifolds whose sets of self-mapping degrees are infinite)? 

We note that, if every simply connected, closed $n$-dimensional manifold is dominated by a product $X_1\times X_2$ such that one of the factors $X_i$ has infinite set of self-mapping degrees, then every finite functorial semi-norm in degree $n$ vanishes on simply connected manifolds, which is an open question by Gromov~\cite[Chapter 5G$_{+}$]{gromovmetric}. 

\bibliographystyle{amsplain}

\begin{thebibliography}{10}

\bibitem{Al}
J.~W.~Alexander, {\em Note on Riemann spaces}, Bull.~Amer.~Math.~Soc.~{\bf 26}, No. 8 (1920), 370--372.

\bibitem{B}
D.~Barden, {\em Simply connected five-manifolds}, Ann.~of~Math.~{\bf 82}, No. 3 (1965), 365--385.

\bibitem{BE}
I.~Berstein and A.~L.~ Edmonds, {\em The degree and branch set of a branched covering}, Invent.~Math.~Vol.~{\bf 45} No. 3 (1978), 213--220.

\bibitem{CL}
D.~Crowley and C.~L\"oh, {\em Functorial semi-norms on singular homology and (in)flexible manifolds}, Preprint: arXiv:1103.4139. 

\bibitem{DSW}
P.~Derbez, H.~Sun and S.~Wang, {\em Finiteness of mapping degree sets for 3-manifolds}, Acta~Math.~Sin.~(Engl.~Ser.) {\bf 27} (2011), 807--812.

\bibitem{Do}
A.~Dold, {\em Erzeugende der Thomschen Algebra $\mathcal{R}$}, Math.~Z.~{\bf 65} (1956), 25--35.

\bibitem{DuL}
H.~Duan and C.~Liang, {\em Circle bundles over 4-manifolds}, Arch. Math., Vol. {\bf 85}, No. 3 (2005), 278--282.

\bibitem{DWa}
H.~Duan and S.~Wang, {\em Non-zero degree maps between $2n$-manifolds}, Acta Math. Sin. (Engl.~Ser.) {\bf 20} (2004), 1--14.

\bibitem{Ed1}
A.~L.~ Edmonds, {\em Deformation of maps to branched coverings in dimension two}, Ann. of Math. (2) {\bf 110}, No. 1 (1979), 113--125.

\bibitem{Ed2}
A.~L.~ Edmonds, {\em Deformation of maps to branched coverings in dimension three}, Math. Ann. {\bf 245} No. 3 (1979), 273--279.

\bibitem{Fre}
M.~H.~Freedman, {\em The topology of four-dimensional manifolds}, J.~Diff.~Geom.~{\bf 17}, No. 3 (1982), 357--453.

\bibitem{Gibl}
P.~J.~Giblin, {\em Circle bundles over a complex quadric}, J.~London Math.~Soc.~{\bf 43} (1968) 323--324.

\bibitem{gromovmetric}
M.~Gromov, {\sl Metric Structures for Riemannian and Non-Riemannian Spaces}, with appendices by M.~Katz, P.~Pansu and S.~Semmes, 
translated from the French by S.~M.~Bates, Progress in Mathematics Vol.~{\bf 152}, Birkh\"auser Verlag 1999.

\bibitem{Ha}
A.~Hatcher, {\em A proof of the Smale conjecture, $\Diff(S^3)\simeq \Or(4)$}, Ann.~of Math.~(2) {\bf 117}, No. 3 (1983), 553--607.


\bibitem{KL}
D.~Kotschick and C.~L\"oh, {\em Fundamental classes not representable by products}, J.~London Math. Soc.~{\bf 79} (2009), 545--561.

\bibitem{KN}
D.~Kotschick and C.~Neofytidis, {\em On three-manifolds dominated by circle bundles}, Math.~Z.~{\bf 274} (2013), 21--32.

\bibitem{NeoThesis}
C.~Neofytidis, {\sl Non-zero degree maps between manifolds and groups presentable by products}, Munich thesis 2014 (available online at: http://edoc.ub.uni-muenchen.de/17204/).

\bibitem{MH}
J.~W.~Milnor and D.~Husemoller, {\sl Symmetric bilinear forms}, Berlin and New York, Springer-Verlag 1973.

\bibitem{PS}
P.~Pankka and J.~Souto, {\em On the non-existence of certain branched covers}, Geom.~Top.~{\bf 16} (2012), 1321--1349.

\bibitem{Ru}
D.~Ruberman, {\em Null-homotopic embedded spheres of codimension one}, Math. Sci. Res. Inst. Publ., {\bf 32}, 229--232, in {\sl Tight and taut
submanifolds}, edited by T.~E.~Cecil and S.-s.~Chern, Cambridge Univ.~Press, Cambridge 1997.

\bibitem{Se}
J.~P.~Serre, {\em Groupes d'homotopie et classes des groupes ab\'eliens}, Ann.~of~Math.~(2) {\bf 58} (1953), 258--294.

\bibitem{Smale}
S.~Smale, {\em Diffeomorphisms of the 2-sphere}, Proc.~Amer.~Math.~Soc.~{\bf 10} (1959), 621--626.

\bibitem{Sm}
S.~Smale, {\em On the structure of 5-manifolds}, Ann.~of~Math.~(2) {\bf 75} (1962), 38--46.

\bibitem{St}
J.~Stallings, {\em A topological proof of Grushko's theorem on free products}, Math.~Z.~{\bf 90} (1965), 1--8.

\bibitem{Steen}
N.~E.~Steenrod, {\em The classification of sphere bundles}, Ann.~of Math.~(2) {\bf 45} (1944), 294--311.

\bibitem{Tan}
L.~Tan, {\em On mapping degrees between 1-connected 5-manifolds},  Sci.~China~Ser. A {\bf 49} No. 12 (2006), 1855--1863.

\bibitem{Wall}
C.~T.~C.~Wall, {\em Diffeomorphisms of 4-manifolds}, J.~London Math.~Soc.~{\bf 39} (1964), 131--140.

\bibitem{Wall0}
C.~T.~C.~Wall, {\em On simply connected 4-manifolds}, J.~London Math.~Soc.~{\bf 39} (1964), 141--149.

\bibitem{Wall1}
C.~T.~C.~Wall, {\em Classification problems in differential topology V. On certain 6-manifolds}, Invent.~Math., Vol. {\bf 1}, No. 4 (1966), 355--374.

\bibitem{Wu}
W.~T.~Wu, {\em Classes charact\'eristiques et i-carres d'une veri\'et\'e}, C.~R.~Acad.~Sci.~Paris, {\bf 230} (1950), 508--509.

\bibitem{Zhubr}
A.~V.~Zhubr, {\em Closed simply connected six-dimensional manifolds: proofs of classification theorems}, (Russian), Algebra i Analiz {\bf 12}  No. 4
(2000), 126--230;  translation in St.~Petersburg Math.~J.~{\bf 12} No. 4 (2001), 605--680.

\end{thebibliography}

\end{document}